\documentclass[compress, preprint,9pt]{elsarticle}

\usepackage[abs]{overpic}
\usepackage{graphicx}
\usepackage{subfigure}
\usepackage{stmaryrd}
\usepackage{amsfonts, color, amssymb}
\usepackage{amsmath}
\usepackage{amsthm}
\usepackage{epsfig}
\usepackage{psfrag}
\usepackage{bm}
\usepackage{paralist}
\usepackage{algorithm}
\usepackage{algorithmicx}
\usepackage{algpseudocode}
\usepackage{appendix}
\usepackage{caption}
\usepackage{hyperref}
\usepackage{color}

\usepackage[misc]{ifsym}

\usepackage[a4paper, total={6in,8in}]{geometry}

\allowdisplaybreaks[4]
\newtheorem{theorem}{Theorem}[section] %
\newtheorem{definition}{Definition}[section] %

\newtheorem{example}{Example}[section]

\newtheorem{lemma}{Lemma}[section]

\usepackage{multirow, booktabs, bigstrut}
\newtheorem{assumption}{Assumption}[section] %
\def\3bar{{|\hspace{-.02in}|\hspace{-.02in}|}}

\allowdisplaybreaks[4]

\allowdisplaybreaks 

\numberwithin{equation}{section}

\journal{Journal of Computational and Applied Mathematics}

\begin{document}

\begin{frontmatter}

\title{Convergence analysis of a weak Galerkin finite element method on a Bakhvalov-type mesh for a singularly perturbed convection-diffusion equation in 2D}

\author[mymainaddress]{Shicheng Liu}
\ead{lsc22@mails.jlu.edu.cn}

\author[mysecondaddress]{Xiangyun Meng}
\ead{xymeng1@bjtu.edu.cn}

\author[mymainaddress]{Qilong Zhai\corref{mycorrespondingauthor}}
\cortext[mycorrespondingauthor]{Corresponding author}
\ead{zhaiql@jlu.edu.cn}

\address[mymainaddress]{School of Mathematics, Jilin University, Changchun {130012}, China}
\address[mysecondaddress]{School of Mathematics and Statistics, Beijing Jiaotong University, Beijing {100044}, China} 

\begin{abstract}
	In this paper, we propose a weak Galerkin finite element method (WG) for solving singularly perturbed convection-diffusion problems on a Bakhvalov-type mesh in 2D. Our method is flexible and allows the use of discontinuous approximation functions on the meshe. An error estimate is devised in a suitable norm and the optimal convergence order is obtained. Finally, numerical experiments are given to support the theory and to show the efficiency of the proposed method.
\end{abstract}

%


\end{frontmatter}


\section{Introduction}
\label{section:introduction}
Consider the following singularly perturbed convection-diffusion problem
\begin{align}
  -\varepsilon \Delta u -\mathbf{b} \cdot \nabla u + cu &= f, \quad \text{in}~\Omega,\label{1.1}\\
  u&=0, \quad \text{on}~\partial\Omega,\label{1.2}
\end{align}
with a positive parameter $\varepsilon$ satisfying $0< \varepsilon \ll 1$, and $\mathbf{b} \in [W^{1, \infty}(\Omega)]^{2}$. The functions $\mathbf{b}$, $c$ and $f$ are assumed to be smooth on $\Omega$. For any $(x, y) \in \overline{\Omega}$, Assume that 
\begin{align}\label{1.3}
  \begin{aligned}
    &b_{1}(x, y) \geq \beta_{1} \geq 0, 
    &&b_{2}(x,y) \geq \beta_{2} \geq 0,
    &&c(x,y) + \frac{1}{2} \nabla \cdot \mathbf{b}(x,y) \geq \gamma \geq 0.
  \end{aligned}
\end{align}
where $\beta_{1}$, $\beta_{2}$ and $\gamma$ are some positive constants. Assumption (\ref{1.3}) makes problem (\ref{1.1})-(\ref{1.2}) has a unique solution in $H^{2}(\Omega) \cup H_{0}^{1}(\Omega)$ for all $f \in L^{2}(\Omega)$, see details in \cite{MR2454024}.

Singularly perturbed problems are one of the important topics in scientific computation. It is well known that the solution of the boundary value problem usually has layers, which are thin regions where the solution or its derivatives change rapidly, due to the diffusion coefficient is very small. In order to resolve the difficulty, numerical stabilization techniques have been developed, which can be divided into fitted operator methods and fitted mesh methods. One of the effective methods of solving singularly perturbed problems is to use layer-adapted meshes. Boundary layers can be resolved by designing layer-adapted meshes if we know a prior knowledge of the layers structure. Commonly used layer-adapted meshes for solving singularly perturbed problems include Bakhvalov-type meshes and Shishkin-type meshes.
Bakhvalov mesh is proposed for the first time in \cite{MR255066}, its application needs a nonlinear equation which cannot be solved explicitly. In order to avoid this difficulty, meshes that arise from an approximation of Bakhvalov's mesh generating function are called Bakhvalov-type meshes, which are one of the most popular layer-adapted meshes, see details in \cite{MR2583792}. Another piecewise equidistant mesh proposed by Shishkin in \cite{Moscow}, but a logarithmic factor will be present in the error bounds when one uses a Shishkin-type mesh. Therefore, Bakhvalov-type meshes have better numerical performance than Shishkin-type meshes in general. 
Even if the layer-adapted meshe is used, the numerical solution of the convection-dominated problem still has some oscillation, as detailed in \cite{MR2785883}. Additional stabilization is added to the numerical scheme to solve these oscillatory behaviour, examples for singularly perturbed onvection-diffusion problem, such as the streamline-diffusion finite element method \cite{MR1811645, MR3813259}, the classical finite difference method of up-winding flavor \cite{MR1690860, MR1825803, MR1849244, MR4440028} and the discontinuous Galerkin methods \cite{MR2452853, MR2495786, MR2355174, MR2194592, MR3419822, MR3143687}.

In this paper, we consider the WG method to solve the singularly perturbed convection-diffusion boundary value problem on Bakhvalov-type mesh. The WG method proves to be an effective numerical technique for the partial differential equations(PDEs). The main idea of this method is that the classical derivative is replaced by weak derivative, which allows the use of discontinuous functions in numerical schemes with parameter independent stabilizers. The initial proposal for its application in solving second-order elliptic problems was made by Junping Wang and Xiu Ye in \cite{wysec2}. The WG method has been applied to all kinds of problems including Stokes equations \cite{wy1302, WangZhaiZhangS2016, wangwangliu2022}, Maxwell's equations \cite{MLW14}, Brinkman equations \cite{MuWangYe14, WangZhaiZhang2016, ZhaiZhangMu16}, fractional time convection-diffusion problems \cite{MR4268647} and so on. For singular perturbed value problems, the WG method has also yielded some results, such as \cite{MR3805854, MR4360590, MR4345844, MR4631818, WG_1D_Bakhvalov, WG_1D_Shishkin, WG_2D_Shishkin, singularly_perturbed_biharmonic}. The main purpose of this paper is to present optimal order uniform convergence in the energy norm on Bakhvalov-type mesh for convection-dominated problems in 2D.

This paper is organized as follows.
In Section 2, we describe the assumptions and introduce a Bakhvalov-type mesh. In Section 3, we introduce the definition of weak operator, the WG scheme and some properties of projection operator involved. In Section 4, we provide convergence analysis. In Section 5, the numerical results verify the correctness of our theory.

\section{Assumption and Partition}
In this section, we will introduce the construction of the Bakhvalov-type mesh and the decomposition of the solution $u$. As in \cite{MR1778398} we shall introduce the following assumptions which describes the structure of $u$.
\begin{assumption}\label{assume2.1}
  For analysis, the solution $u$ of the equation (\ref{1.1}) can can be decomposed as follows
  \begin{align*}
    u=S+E_{1}+E_{2}+E_{12}.
  \end{align*}
  where $S$ represents the smooth part, $E_{1}$ and $E_{2}$ corresponds to boundary layer components, $E_{12}$ is corner layer part. Then, for $0 \leq i+j \leq k+1$, there exists a constant $C$ such taht
  \begin{align}
       &\left\vert\frac{\partial^{i+j}S}{\partial x^{i}\partial y^{j}}\right\vert\leq C, 
      & &\left\vert\frac{\partial^{i+j}E_{1}}{\partial x^{i}\partial y^{j}}\right\vert\leq C\varepsilon^{-j}e^{-\frac{\beta_{2}y}{\varepsilon}},\\
      &\left\vert\frac{\partial^{i+j}E_{2}}{\partial x^{i}\partial y^{j}}\right\vert\leq C\varepsilon^{-i}e^{-\frac{\beta_{1}x}{\varepsilon}},
      & &\left\vert\frac{\partial^{i+j}E_{12}}{\partial x^{i}\partial y^{j}}\right\vert\leq C\varepsilon^{-(i+j)}e^{-\frac{\beta_{1}x+\beta_{2}y}{\varepsilon}}.
  \end{align}
  \end{assumption}

  For the convection-diffusion problem (\ref{1.1})-(\ref{1.2}), we consider the following Bakhvalov-type mesh introduced in \cite{MR2583792}, which is defined by
\begin{align}\label{bakhvalov_mesh}
  \begin{aligned}
  x_{i} = 
    \begin{cases}
      -\frac{\sigma \varepsilon}{\beta_{1}} \ln(1-2(1-\varepsilon)i/N), \quad \text{for} ~ i=0, \cdots, N/2\\
      1-(1-x_{\frac{N}{2}})2(N-i)/N, \quad \text{for} ~ i=\frac{N}{2}+1, \cdots, N 
    \end{cases}\\
  y_{j} = 
    \begin{cases}
      -\frac{\sigma \varepsilon}{\beta_{1}} \ln(1-2(1-\varepsilon)j/N), \quad \text{for} ~ j=0, \cdots, N/2\\
      1-(1-y_{\frac{N}{2}})2(N-j)/N, \quad \text{for} ~ j=\frac{N}{2}+1, \cdots, N 
    \end{cases}
  \end{aligned}
\end{align}
where $N \geq 4$ is an positive integer and $\sigma \geq k+1$, see Figure \ref{figure_1}. Transition points are $x_{N/2}$ and $y_{N/2}$, indicating a shift in mesh from coarse to fine. Then, we get a rectangulation mesh denoted as $\mathcal{T}_{N}$. For the sake of briefness, set
\begin{align*}
  &\Omega_{12} = : [x_{0}, x_{N/2-1}] \times [y_{0}, y_{N/2-1}],
  &&\Omega_{1} = : [x_{N/2-1}, x_{N}] \times [y_{0}, y_{N/2-1}],\\
  &\Omega_{2}  = : [x_{0}, x_{N/2-1}] \times [y_{N/2-1}, y_{N}],
  &&\Omega_{0} = : [x_{N/2-1}, x_{N}] \times [y_{N/2-1}, y_{N}].
\end{align*}
\begin{figure}[h]
  \centering
  \includegraphics[width=12cm]{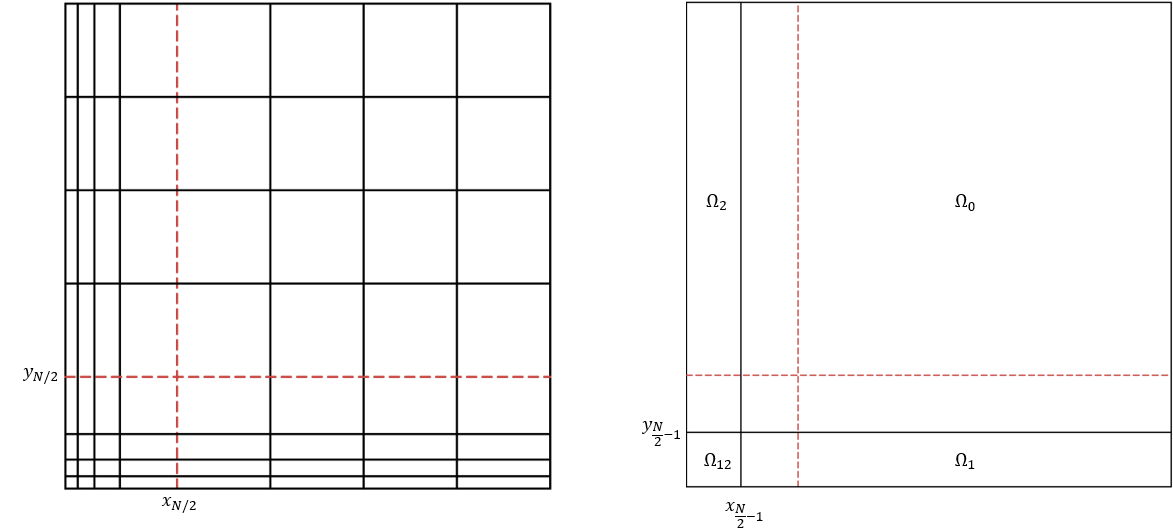}
      \caption{A Bakhvalov-type mesh with $N = 8$ and dissection of $\Omega$.}
      \label{figure_1}
\end{figure}

Let $h_{x,i} = x_{i} - x_{i-1}$ and $h_{y, j} = y_{j} - y_{j-1}$, we omit the subscript $x$ or $y$ if there is no confusion. According to \cite{MR4152300}, we have the following lemma which introduce some important properties of Bakhvalov-type mesh.
\begin{lemma} \label{lemma2.1}
  In this paper, suppose that $\varepsilon \leq N^{-1}$, then for Bakhvalov-type mesh (\ref{bakhvalov_mesh}), one can obtain the properties
  \begin{align}
    & C_{1} \varepsilon N^{-1} \leq h_{1} \leq C_{2} \varepsilon N^{-1},\\
    & h_{1} \leq h_{2} \leq \cdots \leq h_{N/2-1},\\
    & \frac{1}{4} \sigma \varepsilon \leq h_{N/2-1} \leq \sigma \varepsilon,\\
    & \frac{1}{2} \sigma \varepsilon \leq h_{N/2} \leq 2 \sigma N^{-1},\\
    & N^{-1} \leq h_{i} \leq 2 N^{-1}, \quad N/2+1 \leq i \leq N,\\
    & C_{3} \sigma \varepsilon \ln N \leq x_{N/2} \leq C_{4} \sigma \ln N, \quad X_{N/2} \geq C \sigma \varepsilon \vert \ln \varepsilon \vert,\\
    & h_{i}^{\rho} e^{-\beta_{1}x / \varepsilon} \leq C \varepsilon^{\rho} N^{-\rho}, \quad 1 \leq i \leq N/2-1, 0 \leq \rho \leq \sigma.
  \end{align}
\end{lemma}

\section{WG scheme}
In this section,  we introduce the notions of WG method. Let $T \in \mathcal{T}_{N}$ be any element with boundary $\partial T$.
We introduce a weak function $v = \{v_0, v_b\}$ on the element $T$, where $v_0 \in L^2(T)$, and $v_b \in L^{2}(\partial T)$. 
Note that $v_b$ has a single value on each edge $e$. Additionally, component $v_b$ may not necessarily be the same as the trace of $v_0$ on $\partial T$.

For any integer $k \geq 1$, we introduce the space of weak functions
\begin{equation*}
  V_N=\left\{v=\{v_0, v_b\}: v_0\in \mathbb{Q}_{k}(T), v_b\in \mathbb{P}_{k}(e), e \in \partial T, \forall T\in\mathcal{T}_{N}\right\},
\end{equation*}
where $\mathbb{Q}_{k}$ is the space of polynomials which are of degree not exceeding $k$ with respect to each one of the variables $x$ and $y$.
Let $V_{N}^{0}$ represent the subspace of $V_h$ defined by
\begin{equation*}
  V_{N}^{0}=\left\{v \in V_h, v_b|_e=0, e\subset\partial T\cap \partial\Omega\right\}.
\end{equation*}

\begin{definition}
  For any $v \in V_{N}$, a discrete weak gradient $\nabla_{w}$ is defined on $T$ as a unique polynomial $\nabla_{w} v \in [\mathbb{Q}_{k}(T)]^{2}$ satisfying:
  \begin{align}\label{3.1}
    \left(\nabla_{w}v, \mathbf{q}\right)_T=-\left(v_0, \nabla\cdot \mathbf{q}\right)_T+\left\langle v_b, \mathbf{q}\cdot\mathbf{n}\right\rangle_{\partial T}, \quad\forall \mathbf{q}\in \left[\mathbb{Q}_k(T)\right]^{2}. 
  \end{align} 
\end{definition}

\begin{definition}
  For any $v \in V_{N}$, a discrete weak convection divergence $\nabla_{w}^{b} v \in \mathbb{Q}_k(T)$ related to $\mathbf{b}$ is defined on $T$ as a unique polynomial satisfyings
  \begin{align}\label{3.2}
    \left(\nabla_{w,k}^{b} v, \tau  \right)_T=-\left(v_0, \nabla\cdot (\mathbf{b} \varphi) \right)_T+\left\langle v_b, \mathbf{b} \cdot \mathbf{n} \varphi \right\rangle_{\partial T}, \quad\forall \varphi \in \mathbb{Q}_k(T). 
  \end{align}
\end{definition}

The following notations are often used 
\begin{align*}
  &(\varphi, \phi)_{\mathcal{T}_{N}} = \sum_{T \in \mathcal{T}_{N}} (\varphi, \phi)_{T}, 
  &&\langle \varphi, \phi \rangle_{\partial \mathcal{T}_{N}} = \sum_{T \in \mathcal{T}_{N}} \langle \varphi, \phi \rangle_{\partial T}, 
  &&\Vert \varphi \Vert_{\mathcal{T}_{N}}^{2} = \sum_{T \in \mathcal{T}_{N}} \Vert \varphi \Vert_{T}^{2},
\end{align*}
We define $\partial T_{x}$ and $ \partial T_{y}$ as the sets of element edges that are parallel to the $x$ and $y$ axes, respectively. Denote
$\partial_{+}T=\{(x,y) \in \partial T | \mathbf{b}(x,y) \cdot \mathbf{n}(x,y) \leq 0\}$, and $S_{T}$ is the penalization parameter given by
\begin{align*}
  \vartheta_{T}=
  \begin{cases}
    N, &\text{if} ~ T \in \Omega_{0}, \\
    \varepsilon h_{y}^{-1}, &\text{if} ~ T \in \Omega \backslash \Omega_{0}, \text{on} ~ \partial T_{x},\\
    \varepsilon h_{x}^{-1}, &\text{if} ~ T \in \Omega \backslash \Omega_{0}, \text{on} ~ \partial T_{y}.
  \end{cases}
\end{align*}
Then, a WG algorithm is proposed in the following. 

\begin{algorithm}
  Find approximate solution $u_{N} = \{u_{0}, u_{b}\} \in V_{N}^{0}$ satisfying
  \begin{align}\label{3.3}
    A(u_N, v)=(f, v_0), \quad \forall v\in V_{N}^{0}.
  \end{align}
  where
  \begin{align*}
    A(u_N, v)=& A_d(u_{N},v)+A_c(u_{N},v),
  \end{align*}
  and
  \begin{align*}
    A_d(u_{N},v) =& (\nabla_w u_N,\nabla_w v)_{\mathcal{T}_{N}} + s_{d}\left(u_N, v\right),\\
    A_c(u_{N},v) =& (\nabla_w^{b} u_N, v_0)_{\mathcal{T}_{N}} + (c u_0, v_0)_{\mathcal{T}_{N}} +s_{c}\left(u_N, v\right)\\
    s_{d}(u_{N}, v) =& \sum_{i=x, y} \sum_{T\in\mathcal{T}_N} \vartheta_{T} \left\langle u_0-u_b,v_0-v_b \right\rangle_{\partial T_{i}},\\
	  s_{c}(u_{N}, v) =& \sum_{\substack{T\in\mathcal{T}_N\\ e \in \partial_{+} T}}  \left\langle -\mathbf{b} \cdot \mathbf{n} (u_0-u_b), v_0-v_b \right\rangle_{e},
  \end{align*}
\end{algorithm}

\begin{definition}
  From the bilinear form in (\ref{3.3}) , for all $v \in V_{N}$ we can derive an energy norm defined by
  \begin{align}\label{3.4}
    \3bar v \3bar^2 = A_{d}(v, v)+ \Vert v_0\Vert_{\mathcal{T}_{h}} +\Vert \vert \mathbf{b} \cdot \mathbf{n} \vert^{\frac{1}{2}} v_0-v_b \Vert_{\partial{\mathcal{T}_{h}}},
  \end{align}
\end{definition}

\begin{lemma}
  There exists a positive constant $\alpha$, independent of $\varepsilon$, such that for $v \in V_{N}^{0}$
  \begin{equation} \label{3.5}
    A(v, v) \geq \alpha \3bar v \3bar^{2},
  \end{equation}
  where $\alpha = \min\{\gamma, \frac{1}{2}\}$.
\end{lemma}
\begin{proof}
  For any $v \in V_{N}^{0}$, it follows from the definition of the weak divergence (\ref{3.2}) that
  \begin{align} \label{3.6}
    \begin{aligned}
    (\nabla_{w}^{b}v, v_{0})_{\mathcal{T}_{N}} &= (v_{0}, \nabla \cdot (\mathbf{b} v_{0}))_{\mathcal{T}_{N}} + \langle v_{b}, \mathbf{b} \cdot \mathbf{n} v_{0} \rangle_{\partial \mathcal{T}_{N}}\\
    &= -\frac{1}{2} (v_{0}, \nabla \cdot (\mathbf{b} v_{0}))_{\mathcal{T}_{N}} - \frac{1}{2} (v_{0}, \nabla \cdot (\mathbf{b} v_{0}))_{\mathcal{T}_{N}} + \langle v_{b}, \mathbf{b} \cdot \mathbf{n} v_{0} \rangle_{\partial \mathcal{T}_{N}}\\
    &= -\frac{1}{2} (v_{0}, \nabla \cdot (\mathbf{b} v_{0}))_{\mathcal{T}_{N}} + \frac{1}{2} (\mathbf{b} \cdot \nabla v_{0}, v_{0})_{\mathcal{T}_{N}} - \frac{1}{2} \langle v_{0}, \mathbf{b} \cdot \mathbf{n} v_{0} \rangle_{\partial \mathcal{T}_{N}} + \langle v_{b}, \mathbf{b} \cdot \mathbf{n} v_{0} \rangle_{\partial \mathcal{T}_{N}}\\
    &= -\frac{1}{2} (v_{0}, \nabla \cdot \mathbf{b} v_{0})_{\mathcal{T}_{N}} - \frac{1}{2} \langle v_{0} - v_{b}, \mathbf{b} \cdot \mathbf{n} v_{0} \rangle_{\partial \mathcal{T}_{N}} + \frac{1}{2} \langle v_{b}, \mathbf{b} \cdot \mathbf{n} (v_{0} - v_{b}) \rangle_{\partial \mathcal{T}_{N}}\\
    &= -\frac{1}{2} (v_{0}, \nabla \cdot \mathbf{b} v_{0})_{\mathcal{T}_{N}} - \frac{1}{2} \langle \mathbf{b} \cdot \mathbf{n} (v_{0} - v_{b}), v_{0} - v_{b} \rangle_{\partial \mathcal{T}_{N}}
    \end{aligned}
  \end{align}
  Then, with together (\ref{1.3}) and (\ref{3.6}) we have
  \begin{align*}
    A_{c}(v, v) &= -(\nabla_{w}^{b} v, v_{0})_{\mathcal{T}_{N}} + (c v_{0}, v_{0})_{\mathcal{T}_{N}} + \langle - \mathbf{b} \cdot \mathbf{n} (v_{0} - v_{b}), v_{0} - v_{b} \rangle_{\partial_{+} \mathcal{T}_{N}} \\
    &= \left((c+\frac{1}{2} \nabla \cdot \mathbf{b}) v_{0}, v_{0} \right)_{\mathcal{T}_{N}} + \frac{1}{2} \langle \mathbf{b} \cdot \mathbf{n} (v_{0} - v_{b}), v_{0} - v_{b} \rangle_{\partial \mathcal{T}_{N}} + \langle - \mathbf{b} \cdot \mathbf{n} (v_{0} - v_{b}), v_{0} - v_{b} \rangle_{\partial_{+} \mathcal{T}_{N}} \\
    &\geq (\gamma v_{0}, v_{0})_{\mathcal{T}_{N}} + \frac{1}{2} \langle - \mathbf{b} \cdot \mathbf{n} (v_{0} - v_{b}), v_{0} - v_{b} \rangle_{\partial_{+} \mathcal{T}_{N}} + \frac{1}{2} \langle \mathbf{b} \cdot \mathbf{n} (v_{0} - v_{b}), v_{0} - v_{b} \rangle_{\partial \mathcal{T}_{N} \backslash \partial_{+} \mathcal{T}_{N}}\\
    &\geq (\gamma v_{0}, v_{0})_{\mathcal{T}_{N}} + \frac{1}{2} \langle \vert \mathbf{b} \cdot \mathbf{n} \vert (v_{0} - v_{b}), v_{0} - v_{b} \rangle_{\partial \mathcal{T}_{N}}
  \end{align*}
  The above inequality and $A_{d}(v,v) = a(v, v) +s_{d}(v, v)$ yield
  \begin{align*}
    A(v, v) \geq \alpha \3bar v \3bar^{2},
  \end{align*}
  where $\alpha = \min\{\gamma, \frac{1}{2}\}$. The lemma is proved.
\end{proof}

Now, we introduce some $L^{2}$ projection operators for each element $T \in \mathcal{T}_N$ detailed as follows. Let operator $\mathcal{Q}_0$ projects onto $\mathbb{Q}_k(T)$. For each edge $e \in \partial T$, consider the operator $\mathcal{Q}_b$ onto $\mathbb{P}_k(e)$. Finally, we define a projection operator $\mathcal{Q}_N u=\left\{\mathcal{Q}_0 u, \mathcal{Q}_b u \right\}$ for solution $u$ onto the space $V_N$ on each element $T$. Additionally, denote by $\mathbf{Q}_{N}$ the projection operator onto $[\mathbb{Q}_{k}(T)]^{2}$. The following lemma is commutative property of weak gradient.

\begin{lemma} \label{lemma3.2}
  On each element $T \in \mathcal{T}_{N}$, for any $v \in H^2(T)$,
  \begin{align}\label{3.7}
    \nabla_{w}(\mathcal{Q}_{N} u)=\mathbf{Q}_{N}(\nabla u).
  \end{align}
\end{lemma}
\begin{proof}
  For any $\mathbf{q} \in [\mathbb{Q}_{k}(T)]^2$, from the weak gradient definition (\ref{3.1}), we get
  \begin{align*}
    \left(\nabla_{w}\mathcal{Q}_{N} u,\mathbf{q}\right)_{T}&=-\left(\mathcal{Q}_{0}u,\nabla\cdot \mathbf{q}\right)_{T}+\left\langle\mathcal{Q}_{b} u,\mathbf{q}\cdot\mathbf{n}\right\rangle_{\partial T}\\
    &=-\left(u,\nabla\cdot \mathbf{q}\right)_{T}+\left\langle u, \mathbf{q}\cdot\mathbf{n}\right\rangle_{\partial T}\\
    &=\left(\nabla u, \mathbf{q}\right)_{T}\\
    &=\left(\mathbf{Q}_{N}\nabla u, \mathbf{q}\right)_{T}.
  \end{align*}
  This corresponds to equality (\ref{3.7}).
\end{proof}

\section{Error estimate}
In this section, we will obtain an error estimate for the WG finite element approximation $u_{N}$. We next derive the following error equation that the error satisfies.
\begin{lemma}
 Suppose that $u$ is the solutions of (\ref{1.1})-(\ref{1.2}) and $u_{N}$ is the solutions of (\ref{3.3}).  Denote $e_{N} = \mathcal{Q}_{N} u -  u_{N}$. Then for any $v \in V_{N}^{0}$, we have
  \begin{align} \label{4.1}
    A(e_{N},v)=l_{1}(u,v)-l_{2}(u,v)-l_{3}(u,v)+s_d(\mathcal{Q}_{N}u,v)+s_c(\mathcal{Q}_{N}u,v),
\end{align}
where
\begin{align*}
  &l_{1}(u,v)=\left( u-\mathcal{Q}_{0} u, \nabla \cdot (\mathbf{b} v_{0}) \right)_{\mathcal{T}_{h}},\\
  &l_{2}(u,v)=\left\langle u-\mathcal{Q}_{b} u, \mathbf{b} \cdot \mathbf{n} (v_0-v_b) \right\rangle_{\partial \mathcal{T}_{h}},\\
  &l_{3}(u,v)=\varepsilon \left\langle \left(\nabla u-\mathbf{Q}_{h} \nabla u \right) \cdot \mathbf{n}, v_0-v_b\right\rangle_{\partial \mathcal{T}_{h}}.
\end{align*}
\end{lemma}
\begin{proof}
  Using (\ref{3.1}), (\ref{3.7}) and integration by parts, we obtain 
  \begin{align*}
    \left(\nabla_{w}\mathcal{Q}_{N}u, \nabla_{w}v\right)_{\mathcal{T}_{N}} &= \left(\mathbf{Q}_{N} \nabla u, \nabla_{w} v \right)_{\mathcal{T}_{N}}\\
    &= -\left(v_0, \nabla\cdot (\mathbf{Q}_{N} \nabla u)\right)_{\mathcal{T}_{N}} + \left\langle v_b, (\mathbf{Q}_{N} \nabla u) \cdot \mathbf{n}\right\rangle_{\partial \mathcal{T}_{N}}\\
    &=\left(\nabla v_{0}, \mathbf{Q}_{N} \nabla u\right)_{\mathcal{T}_{N}} - \left\langle v_0-v_b, (\mathbf{Q}_{N} \nabla u) \cdot \mathbf{n}\right\rangle_{\partial \mathcal{T}_{N}}\\
    &=\left(\nabla u, \nabla v_{0}\right)_{\mathcal{T}_{N}} - \left\langle v_0-v_b, (\mathbf{Q}_{N} \nabla u) \cdot\mathbf{n}\right\rangle_{\partial \mathcal{T}_{N}}\\
    &=(-\Delta u, v_{0})_{\mathcal{T}_{N}} + \langle \nabla u \cdot \mathbf{n}, v_{0} \rangle_{\partial \mathcal{T}_{N}} - \left\langle v_0-v_b, (\mathbf{Q}_{N} \nabla u) \cdot\mathbf{n}\right\rangle_{\partial \mathcal{T}_{N}} - \langle \nabla u \cdot \mathbf{n}, v_{b} \rangle_{\partial \mathcal{T}_{N}},\\
    &= (-\Delta u, v_{0})_{\mathcal{T}_{N}} - \left\langle (\mathbf{Q}_{N} - \nabla u) \cdot\mathbf{n}, v_0-v_b \right\rangle_{\partial \mathcal{T}_{N}}
\end{align*}
where we have used the fact that $\langle \nabla u \cdot \mathbf{n}, v_{b} \rangle_{\partial \mathcal{T}_{N}} = 0$. Similarly, from (\ref{3.2}) and integration by parts one has
\begin{align*}
  -(\nabla_{w}^{b} \mathcal{Q}_{N} u, v_{0})_{\mathcal{T}_{N}} &= (\mathcal{Q}_{0} u, \nabla \cdot (\mathbf{b} v_{0}))_{\mathcal{T}_{N}} - \langle \mathcal{Q}_{b} u, \mathbf{b} \cdot \mathbf{n} v_{0} \rangle_{\partial \mathcal{T}_{N}}\\
  &= (\mathcal{Q}_{0} u - u, \nabla \cdot (\mathbf{b} v_{0}))_{\mathcal{T}_{N}} + (u, \nabla \cdot (\mathbf{b} v_{0}))_{\mathcal{T}_{N}} - \langle \mathcal{Q}_{b} u, \mathbf{b} \cdot \mathbf{n} v_{0} \rangle_{\partial \mathcal{T}_{N}}\\
  &= (\mathcal{Q}_{0} u - u, \nabla \cdot (\mathbf{b} v_{0}))_{\mathcal{T}_{N}} - (\mathbf{b} \cdot \nabla u, v_{0})_{\mathcal{T}_{N}} - \langle \mathcal{Q}_{b} u - u, \mathbf{b} \cdot \mathbf{n} v_{0} \rangle_{\partial \mathcal{T}_{N}}\\
  &= -(\mathbf{b} \cdot \nabla u, v_{0})_{\mathcal{T}_{N}} + (\mathcal{Q}_{0} u - u, \nabla \cdot (\mathbf{b} v_{0}))_{\mathcal{T}_{N}} - \langle \mathcal{Q}_{b} u - u, \mathbf{b} \cdot \mathbf{n} (v_{0} - v_{b}) \rangle_{\partial \mathcal{T}_{N}}
\end{align*}
where we have used $\langle \mathcal{Q}_{b} u - u, \mathbf{b} \cdot \mathbf{n} v_{b} \rangle_{\partial \mathcal{T}_{N}}=0$. Also we have
\begin{align*}
  (c \mathcal{Q}_{N} u, v_{0})_{\mathcal{T}_{N}} = (c \mathcal{Q}_{0} u ,v_{0})_{\mathcal{T}_{N}} = (c u, v_{0})_{\mathcal{T}_{N}}.
\end{align*}
Notice that
\begin{align*}
  \varepsilon (\nabla_{w} u_{N}, \nabla_{w} v) - (\nabla_{w}^{b} u_{N}, v_{0}) + (c u_{N}, v_{0}) + s_{d}(u_{N}, v) + s_{c}(u_{N}, v) &= (f, v_{0}),
\end{align*}
and
\begin{align*}
  -\varepsilon (\Delta u, v_{0}) - (\mathbf{b} \cdot \nabla u, v_{0}) + (c u, v_{0}) &= (f, v_{0}).
\end{align*}
Combined with the above equality we yield
\begin{align*}
  \varepsilon (\nabla_{w}^{b} \mathcal{Q}_{N} u, \nabla_{w} v) - (\nabla_{w}^{b} \mathcal{Q}_{N} u, v_{0}) + (c \mathcal{Q}_{N} u, v_{0}) = (f, v_{0}) + l_{1}(u, v) - l_{2}(u, v) - l_{3}(u, v).
\end{align*}
By adding $s_{d}(\mathcal{Q}_{N}u, v)$ and $s_{c}(\mathcal{Q}_{N}u, v)$ to both sides of the above equation, we get
\begin{align*}
  A(\mathcal{Q}_{N} u - u_N, v) = l_{1}(u, v) - l_{2}(u, v) - l_{3}(u, v) + s_{d}(\mathcal{Q}_{N}u, v) +s_{c}(\mathcal{Q}_{N}u, v).
\end{align*}
The lemma is proved.
\end{proof}

\begin{theorem}
  Let $\sigma \geq k+1$. Suppose that $u \in H^{k+1}(\Omega)$ and $u_{N}$ be the solutions of (\ref{1.1})-(\ref{1.2}) and $(\ref{3.3})$, respectively. Then one has
  \begin{equation} \label{4.2}
    \3bar \mathcal{Q}_{N} u - u_{N} \3bar \leq C N^{-k}.
  \end{equation}
\end{theorem}
\begin{proof}
  Let $v = e_{N}$ in the error equation (\ref{4.1}), we derive the following equation
  \begin{align*}
    A(e_{N}, e_{N}) = l_{1}(u, e_{N}) - l_{2}(u, e_{N}) - l_{3}(u, e_{N}) + s_{d}(\mathcal{Q}_{N}u, e_{N}) +s_{c}(\mathcal{Q}_{N}u, e_{N}). 
  \end{align*}

  Using the Cauchy-Schwarz inequality and (\ref{A.7}), we obtain 
  \begin{align*}
    \begin{aligned}
      l_{1}(u, e_{N}) &= \left( u-\mathcal{Q}_{0} u, \nabla \cdot (\mathbf{b} e_{0}) \right)_{\mathcal{T}_{h}}\\
      &= \left( \mathcal{Q}_{0} u - u, \nabla \cdot \mathbf{b} e_{0} \right)_{\mathcal{T}_{h}} + \left( \mathcal{Q}_{0} u - u, \mathbf{b} \cdot \nabla e_{0} \right)_{\mathcal{T}_{h}}\\
      &\leq C \left\Vert \mathcal{Q}_{0} u - u \right\Vert_{\mathcal{T}_{h}} \left\Vert e_{0} \right\Vert_{\mathcal{T}_{h}} + \left( \mathcal{Q}_{0} u - u, b_{1} \partial_{x} e_{0} + b_{2} \partial_{y} e_{0} \right)_{\mathcal{T}_{h}}\\
      &\leq C \left\Vert \mathcal{Q}_{0} u - u \right\Vert_{\mathcal{T}_{h}} \left\Vert e_{0} \right\Vert_{\mathcal{T}_{h}} + \left( \mathcal{Q}_{0} u - u, \left( b_{1} - \overline{b_{1}} \right) \partial_{x} e_{0} + \left( b_{2} - \overline{b_{2}} \right) \partial_{y} e_{0} \right)_{\mathcal{T}_{h}}\\
      &\leq C \left\Vert \mathcal{Q}_{0} u - u \right\Vert_{\mathcal{T}_{h}} \left\Vert e_{0} \right\Vert_{\mathcal{T}_{h}}\\
      &\leq C N^{-(k+1)} \3bar e_{N} \3bar.
    \end{aligned}
  \end{align*}
  where $\overline{b_{1}}$ and $\overline{b_{2}}$ are piecewise constant functions whose value are $h_{x}^{-1} \int_{e} b_{1} \,dx $ and $h_{y}^{-1} \int_{e} b_{2} \,dy $, respectively. Similarly, it follows from the Cauchy-Schwarz inequality and (\ref{A.8}) that
  \begin{align*}
    \begin{aligned}
      l_{2}(u, e_{N}) &= \left\langle u-\mathcal{Q}_{b} u, \mathbf{b} \cdot \mathbf{n} (e_0-e_b) \right\rangle_{\partial \mathcal{T}_{h}}\\
      &\leq C \sum_{i=x, y} \sum_{T \in \mathcal{T}_{N}} \left\Vert u - \mathcal{Q}_{b} u \right\Vert_{\partial T_{i}} \left\Vert \left\vert \mathbf{b} \cdot \mathbf{n} \right\vert^{\frac{1}{2}} (e_0-e_b) \right\Vert_{\partial T_{i}}\\
      &\leq C \sum_{i=x, y} \sum_{T \in \mathcal{T}_{N}} \left\Vert u - \mathcal{Q}_{0} u \right\Vert_{\partial T_{i}} \left\Vert \left\vert \mathbf{b} \cdot \mathbf{n} \right\vert^{\frac{1}{2}} (e_0-e_b) \right\Vert_{\partial T_{i}}\\
      &\leq C \sum_{i=x, y} \left\Vert u - \mathcal{Q}_{0} u \right\Vert_{\partial \mathcal{T}_{Ni}} \left\Vert \left\vert \mathbf{b} \cdot \mathbf{n} \right\vert^{\frac{1}{2}} (e_0-e_b) \right\Vert_{\mathcal{T}_{Ni}}\\
      &\leq C N^{-(k+\frac{1}{2})} \3bar e_{N} \3bar.
    \end{aligned}
  \end{align*}
  From the Cauchy-Schwarz inequality and (\ref{A.9}), we yield
  \begin{align*}
    \begin{aligned}
      l_{3}(u, e_{N}) &= \varepsilon \left\langle \left(\nabla u-\mathbf{Q}_{h} \nabla u \right) \cdot \mathbf{n}, e_0 - e_b \right\rangle_{\partial \mathcal{T}_{h}}\\
      &\leq C \varepsilon \sum_{i=x, y} \sum_{T \in \mathcal{T}_{N}} \left\Vert \nabla u-\mathbf{Q}_{h} \nabla u \right\Vert_{\partial T_{i}} \left\Vert e_0-e_b \right\Vert_{\partial T_{i}}\\
      &\leq C \sum_{i=x, y} \left( \sum_{T \in \mathcal{T}_{N}} \theta_{T} \left\Vert \nabla u-\mathbf{Q}_{h} \nabla u \right\Vert_{\partial T_{i}}^{2} \right)^{\frac{1}{2}} \left( \sum_{T \in \mathcal{T}_{N}} \vartheta_{T} \left\Vert e_0-e_b \right\Vert_{\partial T_{i}}^{2} \right)^{\frac{1}{2}}\\
      &\leq C N^{-k} \3bar e_{N} \3bar.
    \end{aligned}
  \end{align*}
  On the other hand, it follows from the definition of stabilization, Cauchy-Schwarz inequality (\ref{A.10}) and (\ref{A.8}) that
  \begin{align*}
    s_{d}(\mathcal{Q}_{N}, e_{N}) &= \sum_{i=x, y} \sum_{T \in \mathcal{T}_N} \vartheta_{T} \left\langle \mathcal{Q}_{0} u - \mathcal{Q}_{b} u,e_0-e_b \right\rangle_{\partial T_{i}}\\
    &\leq C \sum_{i=x, y} \left( \sum_{T \in \mathcal{T}_N} \vartheta_{T} \left\Vert \mathcal{Q}_{0} u - u \right\Vert_{\partial T_{i}}^{2} \right)^{\frac{1}{2}} \left( \sum_{T \in \mathcal{T}_N} \vartheta_{T} \left\Vert e_0-e_b \right\Vert_{\partial T_{i}}^{2} \right)^{\frac{1}{2}}\\
    &\leq C N^{-k} \3bar e_{N} \3bar.
  \end{align*}
  and
  \begin{align*}
    s_{c}(\mathcal{Q}_{N}, e_{N}) &= \sum_{i=x, y} \sum_{T\in\mathcal{T}_N}  \left\langle -\mathbf{b} \cdot \mathbf{n} (\mathcal{Q}_{0} u - \mathcal{Q}_{b} u), v_0-v_b \right\rangle_{\partial_{+} T_{i}}\\
    &\leq C \sum_{i=x, y} \left( \sum_{T \in \mathcal{T}_N} \left\Vert \mathcal{Q}_{0} u - u \right\Vert_{\partial T_{i}}^{2} \right)^{\frac{1}{2}} \left( \sum_{T \in \mathcal{T}_N} \left\vert \mathbf{b} \cdot \mathbf{n} \right\vert^{\frac{1}{2}} \left\Vert e_0-e_b \right\Vert_{\partial T_{i}}^{2} \right)^{\frac{1}{2}}\\
    &\leq C N^{-(k+\frac{1}{2})} \3bar e_{N} \3bar.
  \end{align*}
  Combining all the estimates above with together (\ref{3.5}) that
  \begin{align*}
    \alpha \3bar e_{N} \3bar^{2} &\leq A(e_{N}, e_{N})\\
    &\leq \vert l_{1}(u, e_{N}) \vert + \vert l_{2}(u, e_{N}) \vert + \vert l_{3}(u, e_{N}) \vert + \vert s_{d}(\mathcal{Q}_{N}, e_{N}) \vert + \vert s_{c}(\mathcal{Q}_{N}, e_{N}) \vert\\
    &\leq C N^{-k} \3bar e_{N} \3bar.
  \end{align*}
  which implies (\ref{4.2}). This completes the proof of the theorem.
\end{proof}

\section{Numerical Experiments}
In this section we present numerical experiments that support our theoretical results, which solution exhibits typical exponential layer behavior. And we select $\sigma = 2k$ for creating the Bakhvalov-type mesh. 
\begin{example}\label{Ex5.1}
  \begin{align*}
    -\varepsilon \Delta u - (2 + 2x - y) \partial_{x} u - (3 - x + 2y) \partial_{y} u + u &= f, \quad \text{in}~\Omega = (0, 1)^{2},\\
    u&=0, \quad \text{on}~\partial\Omega,
  \end{align*}
  where $f(x, y)$ is chosen such that the exact solution
  $$u(x, y) = 2 \sin(\pi x) \left( 1 - e^{-\frac{2x}{\varepsilon}} \right) \left( 1 - y \right)^{2} \left( 1 - e^{-\frac{y}{\varepsilon}} \right).$$
  The error $\3bar e_{N} \3bar$ and convergence rates for several different $\varepsilon$ are shown in Table \ref{Ex1_k1}-\ref{Ex1_k2}.
\end{example}

\begin{table}[h]
  \centering
  \caption{Example \ref{Ex5.1}, k=1}
  \begin{tabular}{|c|c|c|c|c|c|c|c|c|c|c|}
    \hline
    N     & \multicolumn{2}{c|}{$\varepsilon$ = 1e-6} & \multicolumn{2}{c|}{$\varepsilon$ = 1e-7} & \multicolumn{2}{c|}{$\varepsilon$ = 1e-8} & \multicolumn{2}{c|}{$\varepsilon$ = 1e-9} & \multicolumn{2}{c|}{$\varepsilon$ = 1e-10} \bigstrut\\
    \hline
    8     & 3.66E-01 & 0.00  & 3.84E-01 & 0.00  & 3.98E-01 & 0.00  & 4.07E-01 & 0.00  & 4.14E-01 & 0.00  \bigstrut\\
    \hline
    16    & 1.73E-01 & 1.08  & 1.83E-01 & 1.07  & 1.91E-01 & 1.06  & 1.96E-01 & 1.05  & 2.00E-01 & 1.05  \bigstrut\\
    \hline
    32    & 8.24E-02 & 1.07  & 8.81E-02 & 1.05  & 9.23E-02 & 1.05  & 9.54E-02 & 1.04  & 9.76E-02 & 1.03  \bigstrut\\
    \hline
    64    & 3.95E-02 & 1.06  & 4.27E-02 & 1.04  & 4.51E-02 & 1.03  & 4.68E-02 & 1.03  & 4.80E-02 & 1.02  \bigstrut\\
    \hline
    128   & 1.90E-02 & 1.06  & 2.08E-02 & 1.04  & 2.21E-02 & 1.03  & 2.30E-02 & 1.02  & 2.37E-02 & 1.02  \bigstrut\\
    \hline
    256   & 9.14E-03 & 1.06  & 1.01E-02 & 1.04  & 1.08E-02 & 1.03  & 1.14E-02 & 1.02  & 1.17E-02 & 1.01  \bigstrut\\
    \hline
  \end{tabular}%
  \label{Ex1_k1}%
\end{table}%

\begin{table}[h]
  \centering
  \caption{Example \ref{Ex5.1}, k=2}
  \begin{tabular}{|c|c|c|c|c|c|c|c|c|c|c|}
    \hline
    N     & \multicolumn{2}{c|}{$\varepsilon$ = 1e-6} & \multicolumn{2}{c|}{$\varepsilon$ = 1e-7} & \multicolumn{2}{c|}{$\varepsilon$ = 1e-8} & \multicolumn{2}{c|}{$\varepsilon$ = 1e-9} & \multicolumn{2}{c|}{$\varepsilon$ = 1e-10} \bigstrut\\
    \hline
    8     & 7.89E-02 & 0.00  & 8.27E-02 & 0.00  & 8.52E-02 & 0.00  & 8.70E-02 & 0.00  & 8.85E-02 & 0.00  \bigstrut\\
    \hline
    16    & 1.93E-02 & 2.03  & 2.04E-02 & 2.02  & 2.11E-02 & 2.02  & 2.16E-02 & 2.01  & 2.20E-02 & 2.01  \bigstrut\\
    \hline
    32    & 4.70E-03 & 2.04  & 5.01E-03 & 2.02  & 5.21E-03 & 2.02  & 5.35E-03 & 2.01  & 5.45E-03 & 2.01  \bigstrut\\
    \hline
    64    & 1.14E-03 & 2.04  & 1.23E-03 & 2.02  & 1.29E-03 & 2.02  & 1.33E-03 & 2.01  & 1.35E-03 & 2.01  \bigstrut\\
    \hline
    128   & 2.77E-04 & 2.05  & 3.02E-04 & 2.03  & 3.18E-04 & 2.02  & 3.29E-04 & 2.01  & 3.36E-04 & 2.01  \bigstrut\\
    \hline
  \end{tabular}%
  \label{Ex1_k2}%
\end{table}%

\begin{example}\label{Ex5.2}
  \begin{align*}
    -\varepsilon \Delta u - 2 \partial_{x} u - 3 \partial_{y} u + u &= f, \quad \text{in}~\Omega = (0, 1)^{2},\\
    u&=0, \quad \text{on}~\partial\Omega,
  \end{align*}
  where $f(x, y)$ is chosen such that the exact solution
  $$u(x, y) = 2 \sin(1 - x) \left( 1 - e^{-\frac{2x}{\varepsilon}} \right) \left( 1 - y \right)^{2} \left( 1 - e^{-\frac{3y}{\varepsilon}} \right).$$
  The error $\3bar e_{N} \3bar$ and convergence rates for several different $\varepsilon$ are shown in Table \ref{Ex2_k1}-\ref{Ex2_k2}.
\end{example}

\begin{table}[h]
  \centering
  \caption{Example \ref{Ex5.2}, k=1}
  \begin{tabular}{|c|c|c|c|c|c|c|c|c|c|c|}
    \hline
    N     & \multicolumn{2}{c|}{$\varepsilon$ = 1e-6} & \multicolumn{2}{c|}{$\varepsilon$ = 1e-7} & \multicolumn{2}{c|}{$\varepsilon$ = 1e-8} & \multicolumn{2}{c|}{$\varepsilon$ = 1e-9} & \multicolumn{2}{c|}{$\varepsilon$ = 1e-10} \bigstrut\\
    \hline
    8     & 8.67E-02 & 0.00  & 8.68E-02 & 0.00  & 8.69E-02 & 0.00  & 8.70E-02 & 0.00  & 8.71E-02 & 0.00  \bigstrut\\
    \hline
    16    & 3.53E-02 & 1.30  & 3.53E-02 & 1.30  & 3.54E-02 & 1.30  & 3.54E-02 & 1.30  & 3.54E-02 & 1.30  \bigstrut\\
    \hline
    32    & 1.50E-02 & 1.24  & 1.50E-02 & 1.24  & 1.50E-02 & 1.24  & 1.50E-02 & 1.24  & 1.50E-02 & 1.24  \bigstrut\\
    \hline
    64    & 6.69E-03 & 1.16  & 6.69E-03 & 1.16  & 6.69E-03 & 1.16  & 6.69E-03 & 1.16  & 6.69E-03 & 1.16  \bigstrut\\
    \hline
    128   & 3.14E-03 & 1.09  & 3.14E-03 & 1.09  & 3.14E-03 & 1.09  & 3.14E-03 & 1.09  & 3.14E-03 & 1.09  \bigstrut\\
    \hline
    256   & 1.51E-03 & 1.05  & 1.51E-03 & 1.05  & 1.51E-03 & 1.05  & 1.51E-03 & 1.05  & 1.51E-03 & 1.05  \bigstrut\\
    \hline
  \end{tabular}%
  \label{Ex2_k1}%
\end{table}%

\begin{table}[h]
  \centering
  \caption{Example \ref{Ex5.2}, k=2}
  \begin{tabular}{|c|c|c|c|c|c|c|c|c|c|c|}
    \hline
    N     & \multicolumn{2}{c|}{$\varepsilon$ = 1e-6} & \multicolumn{2}{c|}{$\varepsilon$ = 1e-7} & \multicolumn{2}{c|}{$\varepsilon$ = 1e-8} & \multicolumn{2}{c|}{$\varepsilon$ = 1e-9} & \multicolumn{2}{c|}{$\varepsilon$ = 1e-10} \bigstrut\\
    \hline
    8     & 1.58E-02 & 0.00  & 1.58E-02 & 0.00  & 1.58E-02 & 0.00  & 1.58E-02 & 0.00  & 1.58E-02 & 0.00  \bigstrut\\
    \hline
    16    & 2.91E-03 & 2.44  & 2.91E-03 & 2.44  & 2.91E-03 & 2.44  & 2.91E-03 & 2.44  & 2.91E-03 & 2.44  \bigstrut\\
    \hline
    32    & 5.83E-04 & 2.32  & 5.83E-04 & 2.32  & 5.83E-04 & 2.32  & 5.83E-04 & 2.32  & 5.83E-04 & 2.32  \bigstrut\\
    \hline
    64    & 1.29E-04 & 2.17  & 1.29E-04 & 2.17  & 1.29E-04 & 2.17  & 1.29E-04 & 2.17  & 1.29E-04 & 2.17  \bigstrut\\
    \hline
    128   & 3.05E-05 & 2.08  & 3.05E-05 & 2.08  & 3.05E-05 & 2.08  & 3.05E-05 & 2.08  & 3.05E-05 & 2.08  \bigstrut\\
    \hline
  \end{tabular}%
  \label{Ex2_k2}%
\end{table}%

Tables \ref{Ex1_k1}–\ref{Ex2_k2} show our numerical results including the errors in the energy norm and convergence rates for Examples \ref{Ex5.1} and \ref{Ex5.2}. These data show uniform convergence of the singular perturbation parameter $\varepsilon$.

\section*{Statements and Declarations}


\smallskip
\noindent
\textbf{Data Availability}.
The code used in this work will be made available upon request to the authors.

\appendix
\section{~}

The goal of this Appendix is to establish some fundamental estimates useful in the error estimate. The following lemmas are employed in the convergence analysis, and readers are directed to \cite{oxford} for a detailed proof process.
\begin{lemma}\label{lemma4.3}
  Consider $\phi \in H^{k+1}(T)$ with $k \geq 1$. Let $\mathcal{Q}_{0} \phi$ denote the $L^2$-projection of $\phi$ onto $\mathbb{Q}_{k}(T)$. Then the following inequality estimate holds,
  \begin{align}
    &\Vert \phi - \mathcal{Q}_{0} \phi \Vert_{T} \leq C \left( h_{x}^{k+1} \left\Vert \partial_{x}^{k+1} \phi \right\Vert_{T} + h_{y}^{k+1} \left\Vert \partial_{y}^{k+1} \phi \right\Vert_{T} \right) \label{A.1} \\
    &\Vert \phi-\mathcal{Q}_{0} \phi\Vert_{\partial T_{i}}\leq C\left(h_{j}^{k+\frac{1}{2}} \Vert \partial_{j}^{k+1}\phi\Vert_{T}+h_{i}^{k+1}h_{j}^{-\frac{1}{2}} \Vert\partial_{i}^{k+1}\phi\Vert_{T}+h_{j}^{\frac{1}{2}}h_{i}^{k} \Vert\partial_{i}^{k}\partial_{j}\phi\Vert_{T}\right), \label{A.2}
  \end{align}
  where $i, j\in\{x, y\}$, and $i\neq j$.
\end{lemma}

\begin{lemma}
  Let $v \in \mathbb{Q}_{k}(T)$ with $k \geq 1$ such that the following inequalities holds,
  \begin{align}
    \Vert v \Vert_{\partial T_{i}} \leq &C h_{j}^{-\frac{1}{2}} \Vert v \Vert_{T}, \label{A.5} \\ 
    \Vert \partial_{i} v \Vert_{T} \leq &C h_{i}^{-1} \Vert v \Vert_{T}, \label{A.6}
  \end{align}
  where $C$ is a constant only depends on $k$ and $i, j \in\{x, y\}$, $i\neq j$. 
\end{lemma}

We would like to establish the following estimates which are useful in the convergence analysis for the WG scheme (\ref{3.3}).
\begin{lemma}
  Let $k \geq 1$ and $u \in H^{k+1}(\Omega)$. There exists a constant $C$ such that the following estimates hold true,
  \begin{align}
    & \left( \sum_{T\in\mathcal{T}_{N}} \left\Vert u - \mathcal{Q}_{0} u \right\Vert_{T}^{2} \right)^{\frac{1}{2}} \leq C N^{-(k
    +1)}, \label{A.7} \\
    & \sum_{i=x, y} \left( \sum_{T\in\mathcal{T}_{N}} \left\Vert u - \mathcal{Q}_{0} u \right\Vert_{\partial T_{i}}^{2} \right)^{\frac{1}{2}} \leq C N^{-(k+\frac{1}{2})}, \label{A.8} \\
    & \sum_{i=x, y} \left( \sum_{T\in\mathcal{T}_{N}} \theta_{T} \left\Vert\nabla u - \mathbf{Q}_{N} \nabla u \right\Vert_{\partial T_{i}}^{2} \right)^{\frac{1}{2}} \leq C N^{-k}. \label{A.9}\\
    & \sum_{i=x, y} \left( \sum_{T\in\mathcal{T}_{N}} \vartheta_{T} \left\Vert u - \mathcal{Q}_{0} u \right\Vert_{\partial T_{i}}^{2} \right)^{\frac{1}{2}} \leq C N^{-k}, \label{A.10}
  \end{align}
  where
  \begin{align*}
    \theta_{T}=
    \begin{cases}
      \varepsilon^{2} N^{-1}, &\text{if} ~ T \in \Omega_{0}, \\
      \varepsilon h_{y}, &\text{if} ~ T \in \Omega \backslash \Omega_{0}, \text{on} ~ \partial T_{x},\\
      \varepsilon h_{x}, &\text{if} ~ T \in \Omega \backslash \Omega_{0}, \text{on} ~ \partial T_{y}.
    \end{cases}
  \end{align*}
\end{lemma}
\begin{proof}
  Each term in the Assumption\ref{assume2.1} will be considered individually. To derive inequality (\ref{A.7}), we use (\ref{A.1}) and Lemma\ref{lemma2.1} to obtain
  \begin{align*}
    \Vert S - \mathcal{Q}_{0} S \Vert_{\Omega_{l}} &\leq C \left(\sum_{T\in \Omega_{l}} \left( h_{x}^{2(k+1)} + h_{y}^{2(k+1)} \right) \cdot h_{x} h_{y} \right)^{\frac{1}{2}}\\
    &\leq C N \left(h_{N}^{2(k+2)} \right)^{\frac{1}{2}}\\
    &\leq C N^{-(k+1)},
  \end{align*}
  Next, considering the boundary layer $E_{1}$ in region $\Omega_{12} \cup \Omega_{1}$, using (\ref{A.1}) and Lemma\ref{lemma2.1} we have
  \begin{align*}
    \Vert E_{1} - \mathcal{Q}_{0} E_{1} \Vert_{\Omega_{12} \cup \Omega_{1}} &\leq C \left(\sum_{T \in \Omega_{12} \cup \Omega_{1}} \left( h_{x}^{2(k+1)} \left\Vert \partial_{x}^{k+1} E_{1} \right\Vert_{T}^{2} + h_{y}^{2(k+1)} \left\Vert \partial_{y}^{k+1} E_{1} \right\Vert_{T}^{2} \right) \right)^{\frac{1}{2}}\\
    &\leq C \left( \sum_{T \in \Omega_{12} \cup \Omega_{1}}\left( h_{x}^{2k+3} h_{y} + \varepsilon^{-2(k+1)} h_{x} h_{y}^{2k+3} \right) \left\Vert e^{-\frac{\beta_{2} y}{\varepsilon}} \right\Vert_{L^{\infty}(T)}^{2}  \right)^{\frac{1}{2}}\\
    &\leq C N \left(h_{x,N}^{2(k+3)} \varepsilon N^{-1} + h_{x,N} h_{y, \frac{N}{2}-1} N^{-2(k+1)} \right)^{\frac{1}{2}}\\
    &\leq C N \left( \varepsilon N^{-2(k+2)} + \varepsilon N^{-(2k+3)} \right)^{\frac{1}{2}}\\
    &\leq C N^{-(k+1)},
  \end{align*}
  As for the region $\Omega_{2} \cup \Omega_{0}$, we have the following estimate
  \begin{align*}
    \Vert E_{1} - \mathcal{Q}_{0} E_{1} \Vert_{\Omega_{2} \cup \Omega_{0}} &\leq C \left( \sum_{T \in \Omega_{2} \cup \Omega_{0}} \left( \left\Vert E_{1} \right\Vert_{T}^{2} + \left\Vert \mathcal{Q}_{0} E_{1} \right\Vert_{T}^{2} \right) \right)^{\frac{1}{2}}\\
    &\leq C \left\Vert E_{1} \right\Vert_{\Omega_{2} \cup \Omega_{0}}\\
    &\leq C N \left( h_{x,N} h_{y,N} \left\Vert E_{1} \right\Vert_{L^{\infty}(\Omega_{2} \cup \Omega_{0})}^{2} \right)^{\frac{1}{2}}\\
    &\leq C N^{-\sigma}
  \end{align*}
  A similar bound can be readily obtained for $E_{2}$. For the concer layer $E_{12}$, by applying inequalities (\ref{A.1}) and Lemma\ref{lemma2.1}, we arrive at
  \begin{align*}
    \Vert E_{12} - \mathcal{Q}_{0} E_{12} \Vert_{\Omega_{12}} &\leq C \left( \sum_{T \in \Omega_{12}} \left( h_{x}^{2k+3} h_{y} \left\Vert \partial_{x}^{k+1} E_{12} \right\Vert_{L^{\infty}(T)}^{2} + h_{x} h_{y}^{2k+3} \left\Vert \partial_{y}^{k+1} E_{12} \right\Vert_{L^{\infty}(T)}^{2} \right) \right)^{\frac{1}{2}}\\
    &\leq C \left( \sum_{T \in \Omega_{12}} \varepsilon^{-2(k+1)} \left( h_{x}^{2k+3} h_{y} + h_{x} h_{y}^{2k+3} \right) \left\Vert e^{-\frac{\beta_{1} x + \beta_{2} y}{\varepsilon}} \right\Vert_{L^{\infty}(T)}^{2}  \right)^{\frac{1}{2}}\\
    &\leq C N \left( \varepsilon N^{-1} N^{-2(k+1)} \left( h_{x, \frac{N}{2}-1} + h_{y, \frac{N}{2}-1} \right) \right)^{\frac{1}{2}}\\
    &\leq C N \left( \varepsilon^{2} N^{-(2k+3)} \right)^{\frac{1}{2}}\\
    &\leq C N^{-(k+\frac{3}{2})},
  \end{align*} 
  Also, we have the estimate in regions $\Omega_{1} \cup \Omega_{0}$ and $\Omega_{2} \cup \Omega_{0}$, as follows
  \begin{align*}
    \Vert E_{12} - \mathcal{Q}_{0} E_{12} \Vert_{\Omega_{1} \cup \Omega_{0}} &\leq C \left( \sum_{T \in \Omega_{1} \cup \Omega_{0}} \left( \left\Vert E_{12} \right\Vert_{T}^{2} + \left\Vert \mathcal{Q}_{0} E_{12} \right\Vert_{T}^{2} \right) \right)^{\frac{1}{2}}\\
    &\leq C \left\Vert E_{12} \right\Vert_{\Omega_{1} \cup \Omega_{0}}\\
    &\leq C N \left( h_{x,N} h_{y,N} \left\Vert E_{12} \right\Vert_{L^{\infty}(\Omega_{1} \cup \Omega_{0})}^{2} \right)^{\frac{1}{2}}\\
    &\leq C N^{-\sigma}
  \end{align*}
  and
  \begin{align*}
    \Vert E_{12} - \mathcal{Q}_{0} E_{12} \Vert_{\Omega_{2} \cup \Omega_{0}} &\leq C \left( \sum_{T \in \Omega_{2} \cup \Omega_{0}} \left( \left\Vert E_{12} \right\Vert_{T}^{2} + \left\Vert \mathcal{Q}_{0} E_{12} \right\Vert_{T}^{2} \right) \right)^{\frac{1}{2}}\\
    &\leq C \left\Vert E_{12} \right\Vert_{\Omega_{2} \cup \Omega_{0}}\\
    &\leq C N \left( h_{x,N} h_{y,N} \left\Vert E_{12} \right\Vert_{L^{\infty}(\Omega_{2} \cup \Omega_{0})}^{2} \right)^{\frac{1}{2}}\\
    &\leq C N^{-\sigma}
  \end{align*} 
  Combining the above inequalities, we have completed the proof of (\ref{A.7}).

  Next, we give the proof of (\ref{A.8}). Let $\Omega_{\ell}$ represents any region in $\Omega_{0}, \Omega_{1}, \Omega_{2}$ and $\Omega_{12}$, by using (\ref{A.2}) and Lemma\ref{lemma2.1} we obtain 
  \begin{align*}
    \left( \sum_{T \in \Omega_{\ell}} \left\Vert S - \mathcal{Q}_{0} S \right\Vert_{\partial T_{i}}^{2} \right)^{\frac{1}{2}} &\leq C \left( \sum_{T \in \Omega_{\ell}} \left( h_{j}^{2k+1} + h_{i}^{2(k+1)} h_{j}^{-1} + h_{i}^{2k} h_{j} \right) \cdot h_{i} h_{j} \right)^{\frac{1}{2}}\\
    &\leq C N \left( h_{N}^{2k+1} \right)^{\frac{1}{2}}\\
    &\leq C N^{-(k + \frac{1}{2})},
  \end{align*}
  For the boundary layer $E_{1}$ in region $\Omega_{12} \cup \Omega_{1}$, using (\ref{A.2}) and Lemma\ref{lemma2.1} we have 
  \begin{align*}
    &\left( \sum_{T \in \Omega_{12} \cup \Omega_{1}} \left\Vert E_{1} - \mathcal{Q}_{0} E_{1} \right\Vert_{\partial T_{x}}^{2} \right)^{\frac{1}{2}} \\
    &\leq C \left( \sum_{T \in \Omega_{12} \cup \Omega_{1}} \left( h_{y}^{2k+1} \left\Vert \partial_{y}^{k+1} E_{1} \right\Vert_{T}^{2} + h_{x}^{2(k+1)} h_{y}^{-1} \left\Vert \partial_{x}^{k+1} E_{1} \right\Vert_{T}^{2} + h_{x}^{2k} h_{y} \left\Vert \partial_{x}^{k} \partial_{y} E_{1} \right\Vert_{T}^{2} \right) \right)^{\frac{1}{2}}\\
    &\leq C \left( \sum_{T \in \Omega_{12} \cup \Omega_{1}} \left( \varepsilon^{-2(k+1)} \cdot h_{x} h_{y}^{2(k+1)} + h_{x}^{2k+3} + \varepsilon^{-2} \cdot h_{x}^{2k+1} h_{y}^{2} \right) \left\Vert e^{-\frac{\beta_{2} y}{\varepsilon}} \right\Vert_{L^{\infty}(T)}^{2} \right)^{\frac{1}{2}}\\
    &\leq C N \left( h_{x, N} N^{-2(k+1)} + h_{x, N}^{2k+3} + h_{x, N}^{2k+1} N^{-2} \right)^{\frac{1}{2}}\\
    &\leq C N^{-(k + \frac{1}{2})},
  \end{align*}
  By using (\ref{A.5}), (\ref{A.6}) and Lemma\ref{lemma2.1} in the rest of the region we get
  \begin{align*}
    \left( \sum_{T \in \Omega_{2} \cup \Omega_{0}} \left\Vert E_{1} - \mathcal{Q}_{0} E_{1} \right\Vert_{\partial T_{x}}^{2} \right)^{\frac{1}{2}} &\leq C \left( \sum_{T \in \Omega_{2} \cup \Omega_{0}} \left( \left\Vert E_{1} \right\Vert_{\partial T_{1}}^{2} + \left\Vert \mathcal{Q}_{0} E_{1} \right\Vert_{\partial T_{1}}^{2} \right) \right)^{\frac{1}{2}}\\
    &\leq C \left( \sum_{T \in \Omega_{2} \cup \Omega_{0}} \left( \left\Vert E_{1} \right\Vert_{\partial T_{1}}^{2} + h_{y}^{-1} \left\Vert E_{1} \right\Vert_{T}^{2} \right) \right)^{\frac{1}{2}}\\
    &\leq C \left( N \int_{0}^{1} e^{-\frac{2 \beta_{1} y}{\varepsilon}} \,dx + \sum_{T \in \Omega_{2} \cup \Omega_{0}} h_{x} \left\Vert e^{-\frac{\beta_{1} y}{\varepsilon}} \right\Vert_{L^{\infty}(T)}^{2} \right)^{\frac{1}{2}}\\
    &\leq C \left( N \cdot N^{-2\sigma} + h_{x, N} \cdot N^{2} \cdot N^{-2\sigma} \right)^{\frac{1}{2}}\\
    &\leq C N^{\frac{1}{2} - \sigma},
  \end{align*}
  Similar result are easily obtained for $\partial T_{y}$. The same goes for the boundary layer $E_{2}$. Now discuss the concer layer $E_{12}$, we have
  \begin{align*}
    &\left( \sum_{T \in \Omega_{12}} \left\Vert E_{12} - \mathcal{Q}_{0} E_{12} \right\Vert_{\partial T_{x}}^{2} \right)^{\frac{1}{2}} \\
    &\leq C \left( \sum_{T \in \Omega_{12}} \left( h_{y}^{2k+1} \left\Vert \partial_{y}^{k+1} E_{12} \right\Vert_{T}^{2} + h_{x}^{2(k+1)} h_{y}^{-1} \left\Vert \partial_{x}^{k+1} E_{12} \right\Vert_{T}^{2} + h_{x}^{2k} h_{y} \left\Vert \partial_{x}^{k} \partial_{y} E_{12} \right\Vert_{T}^{2} \right) \right)^{\frac{1}{2}}\\
    &\leq C \left( \sum_{T \in \Omega_{12}} \varepsilon^{-2(k+1)} \left( h_{x} h_{y}^{2(k+1)} + h_{x}^{2k+3} + h_{x}^{2k+1} h_{y}^{2} \right) \left\Vert e^{-\frac{\beta_{1} x + \beta_{2} y}{\varepsilon}} \right\Vert_{L^{\infty}(T)}^{2} \right)^{\frac{1}{2}}\\
    &\leq C N \left( \varepsilon N^{-1} N^{-2(k+1)} + h_{x, \frac{N}{2}-1} N^{-2(k+1)} + \varepsilon N^{-(2k+1)} N^{-2} \right)^{\frac{1}{2}}\\
    &\leq C N^{-(k + \frac{1}{2})},
  \end{align*}
 where we have used (\ref{A.2}) and Lemma\ref{lemma2.1}. And in other regions, using (\ref{A.5}), (\ref{A.6}) and Lemma\ref{lemma2.1} we derive
  \begin{align*}
    \left( \sum_{T \in \Omega_{1} \cup \Omega_{0}} \left\Vert E_{12} - \mathcal{Q}_{0} E_{12} \right\Vert_{\partial T_{x}}^{2} \right)^{\frac{1}{2}} &\leq C \left( \sum_{T \in \Omega_{1} \cup \Omega_{0}} \left( \left\Vert E_{12} \right\Vert_{\partial T_{x}}^{2} + \left\Vert \mathcal{Q}_{0} E_{12} \right\Vert_{\partial T_{x}}^{2} \right) \right)^{\frac{1}{2}}\\
    &\leq C \left( \sum_{T \in \Omega_{1} \cup \Omega_{0}} \left( \left\Vert E_{12} \right\Vert_{\partial T_{x}}^{2} + h_{y}^{-1} \left\Vert E_{12} \right\Vert_{T}^{2} \right) \right)^{\frac{1}{2}}\\
    &\leq C \left( N \int_{x_{N/2-1}}^{1} e^{-\frac{2 (\beta_{1} x + \beta_{2} y)}{\varepsilon}} \,dx + \sum_{T \in \Omega_{1} \cup \Omega_{0}} h_{x} \left\Vert e^{-\frac{\beta_{1} x + \beta_{2} y}{\varepsilon}} \right\Vert_{L^{\infty}(T)}^{2} \right)^{\frac{1}{2}}\\
    &\leq C \left( N \cdot \varepsilon N^{-2\sigma} + h_{x, N} \cdot N^{2} \cdot N^{-2\sigma} \right)^{\frac{1}{2}}\\
    &\leq C N^{\frac{1}{2} - \sigma},
  \end{align*}
  and
  \begin{align*}
    \left( \sum_{T \in \Omega_{2} \cup \Omega_{0}} \left\Vert E_{12} - \mathcal{Q}_{0} E_{12} \right\Vert_{\partial T_{x}}^{2} \right)^{\frac{1}{2}} &\leq C \left( \sum_{T \in \Omega_{2} \cup \Omega_{0}} \left( \left\Vert E_{12} \right\Vert_{\partial T_{x}}^{2} + \left\Vert \mathcal{Q}_{0} E_{12} \right\Vert_{\partial T_{x}}^{2} \right) \right)^{\frac{1}{2}}\\
    &\leq C \left( \sum_{T \in \Omega_{2} \cup \Omega_{0}} \left( \left\Vert E_{12} \right\Vert_{\partial T_{x}}^{2} + h_{y}^{-1} \left\Vert E_{12} \right\Vert_{T}^{2} \right) \right)^{\frac{1}{2}}\\
    &\leq C \left( N \int_{0}^{1} e^{-\frac{2 (\beta_{1} x + \beta_{2} y)}{\varepsilon}} \,dx + \sum_{T \in \Omega_{2} \cup \Omega_{0}} h_{x} \left\Vert e^{-\frac{\beta_{1} x + \beta_{2} y}{\varepsilon}} \right\Vert_{L^{\infty}(T)}^{2} \right)^{\frac{1}{2}}\\
    &\leq C \left( N \cdot \varepsilon N^{-2\sigma} + h_{x, N} \cdot N^{2} \cdot N^{-2\sigma} \right)^{\frac{1}{2}}\\
    &\leq C N^{\frac{1}{2} - \sigma},
  \end{align*}
  Also it is easy to get similar result for $\partial T_{y}$.
  Together with the above estimates, we easily get inequality (\ref{A.8}).

  Let's prove (\ref{A.9}) now. The notation $\Omega_{\tau}$ represents any region in $\Omega_{0}, \Omega_{1}$, and $\Omega_{2}$. By the definition of $\theta_{T}$ and (\ref{A.2}), we obtain
  \begin{align*}
    \left( \sum_{T \in \Omega_{\tau}} \theta_{T} \left\Vert \nabla S - \mathbf{Q}_{N} \nabla S \right\Vert_{\partial T_{i}}^{2} \right)^{\frac{1}{2}} &= \left( \sum_{T \in \Omega_{\tau}} \varepsilon h_{j} \left\Vert \nabla S - \mathbf{Q}_{N} \nabla S \right\Vert_{\partial T_{i}}^{2} \right)^{\frac{1}{2}}\\
    &\leq C \varepsilon^{\frac{1}{2}} \left( \sum_{T \in \Omega_{12}} \left( h_{j}^{2k} + h_{i}^{2k} + h_{i}^{2(k-1)} h_{j}^{2} \right) \cdot h_{i} h_{j} \right)^{\frac{1}{2}}\\
    &\leq C \varepsilon^{\frac{1}{2}} N \left( h_{N}^{2(k+1)} \right)^{\frac{1}{2}}\\
    &\leq C N^{-(k+\frac{1}{2})},
  \end{align*}
  And for the region $\Omega_{0}$, we also have
  \begin{align*}
    \left( \sum_{T \in \Omega_{0}} \theta_{T} \left\Vert \nabla S - \mathbf{Q}_{N} \nabla S \right\Vert_{\partial T_{i}}^{2} \right)^{\frac{1}{2}} &= \left( \sum_{T \in \Omega_{0}} \varepsilon^{2} N^{-1} \left\Vert \nabla S - \mathbf{Q}_{N} \nabla S \right\Vert_{\partial T_{i}}^{2} \right)^{\frac{1}{2}}\\
    &\leq C \varepsilon \left( \sum_{T \in \Omega_{12}} N^{-1} \left( h_{j}^{2k-1} + h_{i}^{2k} h_{j}^{-1} + h_{i}^{2(k-1)} h_{j} \right) \cdot h_{i} h_{j} \right)^{\frac{1}{2}}\\
    &\leq C \varepsilon N \left( N^{-1} \cdot h_{N}^{2k} \right)^{\frac{1}{2}}\\
    &\leq C N^{-(k+\frac{1}{2})},
  \end{align*}
  where $i, j \in \{x, y\}$ and $i \neq j$. 
  For the boundary layer $E_{1}$ in region $\Omega_{12} \cup \Omega_{1}$, we have
  \begin{align*}
    &\left( \sum_{T \in \Omega_{12} \cup \Omega_{1}} \theta_{T} \left\Vert \nabla E_{1} - \mathbf{Q}_{N} \nabla E_{1} \right\Vert_{\partial T_{x}}^{2} \right)^{\frac{1}{2}} = \left( \sum_{T \in \Omega_{12} \cup \Omega_{1}} \varepsilon h_{y} \left\Vert \nabla E_{1} - \mathbf{Q}_{N} \nabla E_{1} \right\Vert_{\partial T_{x}}^{2} \right)^{\frac{1}{2}}\\
    &\leq C \varepsilon^{\frac{1}{2}}   \Bigg(   \sum_{T \in \Omega_{12} \cup \Omega_{1}}   \Big(   h_{y}^{2k} \left( \left\Vert \partial_{y}^{k} \partial_{x} E_{1} \right\Vert_{T}^{2} + \left\Vert \partial_{y}^{k+1} E_{1} \right\Vert_{T}^{2} \right) + h_{x}^{2k} \left( \left\Vert \partial_{x}^{k+1} E_{1} \right\Vert_{T}^{2} + \left\Vert \partial_{x}^{k} \partial_{y} E_{1} \right\Vert_{T}^{2} \right) \\
    & \quad + h_{x}^{2(k-1)} h_{y}^{2} \left( \left\Vert \partial_{x}^{k} \partial_{y} E_{1} \right\Vert_{T}^{2} + \left\Vert \partial_{x}^{k-1} \partial_{y}^{2} E_{1} \right\Vert_{T}^{2} \right)  \Big)   \Bigg)^{\frac{1}{2}}\\
    &\leq C \varepsilon^{\frac{1}{2}}   \Bigg(   \sum_{T \in \Omega_{12} \cup \Omega_{1}}   \Big(   h_{y}^{2k+1} h_{x} \left( \varepsilon^{-2k} + \varepsilon^{-2(k+1)} \right) + h_{x}^{2k+1} h_{y} \left( 1 + \varepsilon^{-2} \right) \\
    & \quad + h_{x}^{2k-1} h_{y}^{3} \left( \varepsilon^{-2} + \varepsilon^{-4} \right)   \Big)   \left\Vert e^{-\frac{\beta_{2} y}{\varepsilon}} \right\Vert_{L^{\infty}(T)}^{2}  \Bigg)^{\frac{1}{2}}\\
    &\leq C \varepsilon^{\frac{1}{2}} N \left( \left( h_{x, N} N^{-(2k+1)} + h_{x, N}^{2k+1} N^{-1} + h_{x, N}^{2k-1} N^{-3} \right) \cdot \left( \varepsilon + \varepsilon^{-1} \right) \right)^{\frac{1}{2}}\\
    &\leq C N^{-k},
  \end{align*}
  In region $\Omega_{2}$, using (\ref{A.5}), (\ref{A.6}) and Lemma\ref{lemma2.1} we get
  \begin{align*}
    &\left( \sum_{T \in \Omega_{2}} \theta_{T} \left\Vert \nabla E_{1} - \mathbf{Q}_{N} \nabla E_{1} \right\Vert_{\partial T_{x}}^{2} \right)^{\frac{1}{2}} = \left( \sum_{T \in \Omega_{2}} \varepsilon h_{y} \left\Vert \nabla E_{1} - \mathbf{Q}_{N} \nabla E_{1} \right\Vert_{\partial T_{x}}^{2} \right)^{\frac{1}{2}}\\
    &\leq C \varepsilon^{\frac{1}{2}} \left( \sum_{T \in \Omega_{2}} h_{y} \left( \left\Vert \partial_{x} E_{1} \right\Vert_{\partial T_{1}}^{2} + \left\Vert \partial_{y} E_{1} \right\Vert_{\partial T_{1}}^{2} \right) + \left\Vert \partial_{x} E_{1} \right\Vert_{T}^{2} + \left\Vert \partial_{y} E_{1} \right\Vert_{T}^{2} \right)^{\frac{1}{2}}\\
    &\leq C \varepsilon^{\frac{1}{2}} \left(  h_{y, N} \cdot N \left( 1 + \varepsilon^{-2}  \right) \int_{0}^{x_{N/2-1}} e^{-\frac{2 \beta_{2} y}{\varepsilon}} \,dx + \sum_{T \in \Omega_{2}} \left( 1 + \varepsilon^{-2} \right) h_{x} h_{y} \left\Vert e^{-\frac{\beta_{2} y}{\varepsilon}} \right\Vert_{L^{\infty}(T)}^{2} \right)^{\frac{1}{2}}\\
    &\leq C \varepsilon^{\frac{1}{2}} \left( \left( 1 + \varepsilon^{-2} \right) \left( \varepsilon \cdot N^{1 - 2 \sigma} + \varepsilon N^{-1} \cdot N^{2 - 2 \sigma} \right) \right)^{\frac{1}{2}}\\
    &\leq C N^{\frac{1}{2} - \sigma},
  \end{align*}
  Similarly, for the region $\Omega_{0}$ we have
  \begin{align*}
    &\left( \sum_{T \in \Omega_{0}} \theta_{T} \left\Vert \nabla E_{1} - \mathbf{Q}_{N} \nabla E_{1} \right\Vert_{\partial T_{x}}^{2} \right)^{\frac{1}{2}} = \left( \sum_{T \in \Omega_{0}} \varepsilon^{2} N^{-1} \left\Vert \nabla E_{1} - \mathbf{Q}_{N} \nabla E_{1} \right\Vert_{\partial T_{x}}^{2} \right)^{\frac{1}{2}}\\
    &\leq C \varepsilon \left( \sum_{T \in \Omega_{0}} N^{-1} \left( \left\Vert \partial_{x} E_{1} \right\Vert_{\partial T_{x}}^{2} + \left\Vert \partial_{y} E_{1} \right\Vert_{\partial T_{x}}^{2} \right) + N^{-1} h_{y}^{-1} \left( \left\Vert \partial_{x} E_{1} \right\Vert_{T}^{2} + \left\Vert \partial_{y} E_{1} \right\Vert_{T}^{2} \right) \right)^{\frac{1}{2}}\\
    &\leq C \varepsilon \left( N^{-1} \cdot N \left( 1 + \varepsilon^{-2} \right) \int_{x_{N/2-1}}^{1} e^{-\frac{2 \beta_{2} y}{\varepsilon}} \,dx + \sum_{T \in \Omega_{2}} \left( 1 + \varepsilon^{-2} \right) N^{-1} h_{x} \left\Vert e^{-\frac{\beta_{2} y}{\varepsilon}} \right\Vert_{L^{\infty}(T)}^{2} \right)^{\frac{1}{2}}\\
    &\leq C \varepsilon \left( \left( 1 + \varepsilon^{-2} \right) N^{-2 \sigma} \right)^{\frac{1}{2}}\\
    &\leq C N^{- \sigma},
  \end{align*}
  It is easy to get similar result for $\partial T_{y}$ and the boundary layer $E_{2}$ in the same way.
  By arguments similar to the ones used above, for the concer layer $E_{12}$, one has
  \begin{align*}
    &\left( \sum_{T \in \Omega_{12}} \theta_{T} \left\Vert \nabla E_{12} - \mathbf{Q}_{N} \nabla E_{12} \right\Vert_{\partial T_{x}}^{2} \right)^{\frac{1}{2}} = \left( \sum_{T \in \Omega_{12} \cup \Omega_{1}} \varepsilon h_{y} \left\Vert \nabla E_{12} - \mathbf{Q}_{N} \nabla E_{12} \right\Vert_{\partial T_{x}}^{2} \right)^{\frac{1}{2}}\\
    &\leq C \varepsilon^{\frac{1}{2}} \left( \sum_{T \in \Omega_{12}} \left( h_{y}^{2k} \left( \left\Vert \partial_{y}^{k} \partial_{x} E_{12} \right\Vert_{T}^{2} + \left\Vert \partial_{y}^{k+1} E_{12} \right\Vert_{T}^{2} \right) + h_{x}^{2k} \left( \left\Vert \partial_{x}^{k+1} E_{12} \right\Vert_{T}^{2} + \left\Vert \partial_{x}^{k} \partial_{y} E_{12} \right\Vert_{T}^{2} \right) + h_{x}^{2(k-1)} h_{y}^{2} \left( \left\Vert \partial_{x}^{k} \partial_{y} E_{12} \right\Vert_{T}^{2} + \left\Vert \partial_{x}^{k-1} \partial_{y}^{2} E_{12} \right\Vert_{T}^{2} \right) \right) \right)^{\frac{1}{2}}\\
    &\leq C \varepsilon^{\frac{1}{2}} \left( \sum_{T \in \Omega_{12}} \varepsilon^{-2(k+1)} \left( h_{y}^{2k+1} h_{x} + h_{x}^{2k+1} h_{y} + h_{x}^{2k-1} h_{y}^{3} \right) \left\Vert e^{-\frac{\beta_{1} x + \beta_{2} y}{\varepsilon}} \right\Vert_{L^{\infty}(T)}^{2} \right)^{\frac{1}{2}}\\
    &\leq C \varepsilon^{\frac{1}{2}} N \left( N^{-(2k+1)} \cdot N^{-1} + N^{-(2k+1)} N^{-1} + N^{-(2k-1)} N^{-3} \right)^{\frac{1}{2}}\\
    &\leq C \varepsilon^{\frac{1}{2}} N^{-k},
  \end{align*}
  Using (\ref{A.5}), (\ref{A.6}) and Lemma\ref{lemma2.1}, we obtain
  \begin{align*}
    &\left( \sum_{T \in \Omega_{1}} \theta_{T} \left\Vert \nabla E_{12} - \mathbf{Q}_{N} \nabla E_{12} \right\Vert_{\partial T_{x}}^{2} \right)^{\frac{1}{2}} = \left( \sum_{T \in \Omega_{1}} \varepsilon h_{2} \left\Vert \nabla E_{12} - \mathbf{Q}_{N} \nabla E_{12} \right\Vert_{\partial T_{x}}^{2} \right)^{\frac{1}{2}}\\
    &\leq C \varepsilon^{\frac{1}{2}} \left( \sum_{T \in \Omega_{1}} h_{y} \left( \left\Vert \partial_{x} E_{12} \right\Vert_{\partial T_{x}}^{2} + \left\Vert \partial_{y} E_{12} \right\Vert_{\partial T_{x}}^{2} \right) + \left\Vert \partial_{x} E_{12} \right\Vert_{T}^{2} + \left\Vert \partial_{y} E_{12} \right\Vert_{T}^{2} \right)^{\frac{1}{2}}\\
    &\leq C \varepsilon^{\frac{1}{2}} \left(  h_{y, \frac{N}{2}-1} \cdot N \varepsilon^{-2} \int_{x_{N/2-1}}^{1} e^{-\frac{2 (\beta_{1} x + \beta_{2} y)}{\varepsilon}} \,dx + \sum_{T \in \Omega_{1}} \varepsilon^{-2} h_{x} h_{y} \left\Vert e^{-\frac{\beta_{1} x + \beta_{2} y}{\varepsilon}} \right\Vert_{L^{\infty}(T)}^{2} \right)^{\frac{1}{2}}\\
    &\leq C \varepsilon^{\frac{1}{2}} \left( \varepsilon^{-2} \left( \varepsilon^{2} \cdot N^{1 - 2 \sigma} + \varepsilon N^{-1} \cdot N^{2 - 2 \sigma} \right) \right)^{\frac{1}{2}}\\
    &\leq C N^{\frac{1}{2} - \sigma},
  \end{align*}
  In the same way, we also have
  \begin{align*}
    &\left( \sum_{T \in \Omega_{2}} \theta_{T} \left\Vert \nabla E_{12} - \mathbf{Q}_{N} \nabla E_{12} \right\Vert_{\partial T_{x}}^{2} \right)^{\frac{1}{2}} = \left( \sum_{T \in \Omega_{2}} \varepsilon h_{2} \left\Vert \nabla E_{12} - \mathbf{Q}_{N} \nabla E_{12} \right\Vert_{\partial T_{x}}^{2} \right)^{\frac{1}{2}}\\
    &\leq C \varepsilon^{\frac{1}{2}} \left( \sum_{T \in \Omega_{2}} h_{y} \left( \left\Vert \partial_{x} E_{12} \right\Vert_{\partial T_{x}}^{2} + \left\Vert \partial_{y} E_{12} \right\Vert_{\partial T_{x}}^{2} \right) + \left\Vert \partial_{x} E_{12} \right\Vert_{T}^{2} + \left\Vert \partial_{y} E_{12} \right\Vert_{T}^{2} \right)^{\frac{1}{2}}\\
    &\leq C \varepsilon^{\frac{1}{2}} \left( h_{y, N} \cdot N \varepsilon^{-2} \int_{0}^{x_{N/2-1}} e^{-\frac{2 (\beta_{1} x + \beta_{2} y)}{\varepsilon}} \,dx + \sum_{T \in \Omega_{2}} \varepsilon^{-2} h_{x} h_{y} \left\Vert e^{-\frac{\beta_{1} x + \beta_{2} y}{\varepsilon}} \right\Vert_{L^{\infty}(T)}^{2} \right)^{\frac{1}{2}}\\
    &\leq C \varepsilon^{\frac{1}{2}} \left( \varepsilon^{-2} \left( \varepsilon \cdot N^{- 2 \sigma} + \varepsilon N^{-1} \cdot N^{2 - 2 \sigma} \right) \right)^{\frac{1}{2}}\\
    &\leq C N^{\frac{1}{2} - \sigma},
  \end{align*}
  and
  \begin{align*}
    &\left( \sum_{T \in \Omega_{0}} \theta_{T} \left\Vert \nabla E_{12} - \mathbf{Q}_{N} \nabla E_{12} \right\Vert_{\partial T_{x}}^{2} \right)^{\frac{1}{2}} = \left( \sum_{T \in \Omega_{0}} \varepsilon^{2} N^{-1} \left\Vert \nabla E_{12} - \mathbf{Q}_{N} \nabla E_{12} \right\Vert_{\partial T_{x}}^{2} \right)^{\frac{1}{2}}\\
    &\leq C \varepsilon \left( \sum_{T \in \Omega_{0}} N^{-1} \left( \left\Vert \partial_{x} E_{12} \right\Vert_{\partial T_{x}}^{2} + \left\Vert \partial_{y} E_{12} \right\Vert_{\partial T_{x}}^{2} \right) + N^{-1} h_{y}^{-1} \left( \left\Vert \partial_{x} E_{12} \right\Vert_{T}^{2} + \left\Vert \partial_{y} E_{12} \right\Vert_{T}^{2} \right) \right)^{\frac{1}{2}}\\
    &\leq C \varepsilon \left( N^{-1} \cdot N \varepsilon^{-2} \int_{x_{N/2-1}}^{1} e^{-\frac{2 (\beta_{1} x + \beta_{2} y)}{\varepsilon}} \,dx + \sum_{T \in \Omega_{2}}  \varepsilon^{-2} N^{-1} h_{x} \left\Vert e^{-\frac{\beta_{1} x + \beta_{2} y}{\varepsilon}} \right\Vert_{L^{\infty}(T)}^{2} \right)^{\frac{1}{2}}\\
    &\leq C \varepsilon \left( \varepsilon^{-2} \left( \varepsilon N^{-2 \sigma} +N^{-2 \sigma} \right) \right)^{\frac{1}{2}}\\
    &\leq C N^{- \sigma},
  \end{align*}
  Also it is easy to get similar result for $\partial T_{y}$. Combining the above estimates, We have completed the proof of (\ref{A.9}).

  Finally, for the last inequality (\ref{A.10}), note that $\vartheta_{T} = \varepsilon h_{j}^{-1} \leq C \varepsilon h_{1}^{-1} \leq C N$ in regions $\Omega_{1}$, $\Omega_{2}$, $\Omega_{12}$ and $\vartheta_{T} = N$ in region $\Omega_{0}$. Together with (\ref{A.8}) we get
  \begin{align*}
    \left( \sum_{T\in\mathcal{T}_{N}} \vartheta_{T} \left\Vert u - \mathcal{Q}_{0} u \right\Vert_{\partial T_{i}}^{2} \right)^{\frac{1}{2}} 
    &\leq C N^{-k}.
  \end{align*}
  At this point, all the estimates are proved.

\end{proof}

\bibliography{library}
\bibliographystyle{siam}


\newpage

\end{document}